\documentclass[12pt,a4paper,twoside]{amsart}

\usepackage{amsmath,amsfonts,amsthm,amsopn,color,amssymb,enumitem,cite,soul}
\usepackage{palatino}
\usepackage{graphicx}
\usepackage[normalem]{ulem}
\usepackage[colorlinks=true,urlcolor=blue,citecolor=red,linkcolor=blue,linktocpage,pdfpagelabels,bookmarksnumbered,bookmarksopen]{hyperref}
\hypersetup{urlcolor=blue, citecolor=red, linkcolor=blue}
\usepackage[left=2.61cm,right=2.61cm,top=2.72cm,bottom=2.72cm]{geometry}

\usepackage[hyperpageref]{backref}

\usepackage[colorinlistoftodos]{todonotes}
\makeatletter
\providecommand\@dotsep{5}
\def\listtodoname{List of Todos}
\def\listoftodos{\@starttoc{tdo}\listtodoname}
\makeatother

\newcommand{\R}{\mathbb{R}}

\newcommand{\RT}{{\mathbb{R}^3}}

\newcommand{\de}{\partial}

 \DeclareMathOperator{\dv}{div}

\renewcommand{\le}{\leslant}
\renewcommand{\ge}{\geslant}
\renewcommand{\a }{\alpha }

\renewcommand{\b }{\beta }

\renewcommand{\d }{\delta }

\renewcommand{\l }{\lambda}
\renewcommand{\ln }{\lambda_n}
\newcommand{\n }{\nabla }

\newcommand{\G}{\Gamma}

\newcommand{\X}{\mathcal{X}}

\renewcommand{\H}{H^1(\RT)}
\newcommand{\Hr}{H^1_r(\RT)}
\newcommand{\HT}{H^1(\RT)}
\newcommand{\HTr}{H^1_r(\RT)}

\newcommand{\N}{\mathbb{N}}

\newcommand{\D }{{\mathcal D}^{1,2}(\RT)}

\newcommand{\irt }{\int_{\RT}}

\def\bbm[#1]{\mbox{\boldmath $#1$}}
\newcommand{\beq }{\begin{equation}}
\newcommand{\eeq }{\end{equation}}

\renewcommand{\le}{\leqslant}
\renewcommand{\ge}{\geqslant}
\newcommand{\dis}{\displaystyle}

\newcommand{\ird }

\newtheorem{theorem}{Theorem}[section]
\newtheorem{lemma}[theorem]{Lemma}

\newtheorem{definition}[theorem]{Definition}
\newtheorem{proposition}[theorem]{Proposition}
\newtheorem{remark}[theorem]{Remark}

\title[On the Schr\"odinger-Born-Infeld system]{On the Schr\"odinger-Born-Infeld system}

\author[A. Azzollini]{Antonio Azzollini}
\address{Dipartimento di Matematica, Informatica ed Economia, Universit\`a degli
Studi della Basilicata,
\newline\indent
Via dell'Ateneo Lucano 10, I-85100
Potenza, Italy}
\email{antonio.azzollini@unibas.it}

\author[A. Pomponio]{Alessio Pomponio}
\address{Dipartimento di Meccanica, Matematica e Management,
Politecnico di Bari
\newline\indent
Via Orabona 4,  70125  Bari, Italy}
\email{alessio.pomponio@poliba.it}

\author[G. Siciliano]{Gaetano Siciliano}
\address{Departamento de Matem\'atica, Instituto de Matem\'atica e Estat\'istica
Universidade de S\~ao Paulo
\newline\indent
Rua do Mat\~ao 1010,  05508-090, S\~ao Paulo, SP, Brazil }
\email{sicilian@ime.usp.br}

\thanks{A. Azzollini and A. Pomponio are partially supported by a grant of the group GNAMPA of INDAM. A. Pomponio is partially supported also by FRA2016 of Politecnico di Bari.
G. Siciliano is supported by Capes, CNPq and Fapesp, Brazil.}
\subjclass[2010]{35J50, 35J93, 35Q60}
\keywords{Schr\"odinger-Born-Infeld equation, nonlinear electromagnetic theory}

\begin{document}

\maketitle

\begin{abstract}
In this paper we study a system which we propose as a model to describe the interaction between matter and electromagnetic field from a dualistic point of view. This system arises from a suitable coupling of the Schr\"odinger and the Born-Infeld lagrangians, this latter replacing the role that, classically, is played by the Maxwell lagrangian.\\
We use a variational approach to find an electrostatic radial ground state solution by means of suitable estimates on the functional of the action.
\end{abstract}

\section{Introduction}

In the recent years, several models have been proposed to provide a mathematical description of the interaction between a charged particle and the electromagnetic field generated by itself. According to two different philosophycal concepts, the way to perform a mathematical formulation can follow two different and, in some way, antithetical approaches.

The theory developed by Born and Infeld (see \cite{BInat} and \cite{BI}) introduced the idea that both the matter and the electromagnetic field were expression of a unique physical entity. According to this unitarian point of view, the system giving a complete description of the dynamics arose variationally starting from a nonlinear version of the Maxwell lagrangian. This unitarian approach was also taken up by Benci and Fortunato in \cite{BF3} (see also \cite{ABDF} and \cite{DS}).\\
On the other hand there is the dualistic point of view, based on the idea that the dynamics can be described coupling equations related with particles and equations related with the electromagnetic field through a suitable combination of the lagrangians.
Starting from the results obtained by Benci and Fortunato \cite{BF1}, the literature is rich of papers studying models based on this latter point of view. \\ In the past, the duality matter-electromagnetic field was usually carried out by means of either Schr\"odinger or Klein-Gordon lagrangian to provide the mathematical description of the particle, and of the Maxwell lagrangian, or higher order approximations (in the sense of Taylor series) of the Born-Infeld lagrangian (see for example \cite{DP} and \cite{BK}) to represent the electromagnetic field.

Recently, Yu has proposed in \cite{yu} a dualistic model obtained coupling Klein-Gordon and Born-Infeld lagrangians and has studied the electrostatic case expressed by the following system
\begin{equation*}\label{eqq}\tag{$\mathcal{KGBI}$}
\begin{cases}
-\Delta u+(m^2-(\omega+\phi)^2)u= |u|^{p-1}u & \hbox{ in }\RT,
\\[3mm]
\dv \left(\dfrac{\n \phi}{\sqrt{1-|\n \phi|^2}}\right)=u^2(\omega + \phi) & \hbox{ in }\RT,
\\
u(x)\to 0, \ \phi(x)\to 0, & \hbox{ as }x\to \infty.
\end{cases}
\end{equation*}

As a consequence of the form of the differential operator in the second equation, a variational approach to the problem can not be performed in the usual functional spaces. In particular, the quantity $1/\sqrt{1-|\n\phi(x)|^2}$ makes sense when $x\in\RT$ is such that $|\n \phi(x)|<1$, being this inequality a necessary constraint to be considered in the functional setting.

Inspired by \cite{yu}, our aim is to propose and study a new model which represents a variant of the well-known Schr\"odinger-Maxwell system as it was introduced in \cite{DM}. Indeed we replace the usual Maxwell lagrangian with the Born-Infeld one
 and we look for the electrostatic solutions. The system in this case becomes
\begin{equation}\label{eq}\tag{$\mathcal{SBI}$}
\begin{cases}
-\Delta u+u+\phi u= |u|^{p-1}u & \hbox{ in }\RT,
\\[3mm]
-\dv \left(\dfrac{\n \phi}{\sqrt{1-|\n \phi|^2}}\right)=u^2 & \hbox{ in }\RT,
\\
u(x)\to 0, \ \phi(x)\to 0, & \hbox{ as }x\to \infty,
\end{cases}
\end{equation}
and we will refer to it as  Schr\"odinger-Born-Infeld system.

At least formally, the system \eqref{eq} comes variationally from the action functional  $F$ defined by
\begin{equation*}
F(u,\phi)=\frac 12 \irt \left(|\n u|^2+u^2\right)
+\frac 12 \irt \phi u^2
-\frac 1{p+1}\irt |u|^{p+1}
-\frac 12 \irt\left(1-\sqrt{1-|\n \phi|^2}\right).
\end{equation*}
Dealing with this functional presents evident difficulties for several reasons, starting with the definition of the functional setting. Indeed we observe that, being on the one hand natural to consider $u\in\HT$, on the other the presence of the term $ \irt\left(1-\sqrt{1-|\n \phi|^2}\right)$ forces us to restrict the setting of admissible functions $\phi$. \\
We define
\begin{equation}\label{eq:spaceX}
\X:=\D \cap \{\phi \in C^{0,1}(\mathbb R^{3}): \| \nabla\phi\|_{\infty} \leq 1\}
\end{equation}
where $\D$ is the completion of $C_c^\infty (\RT)$ with respect to the norm $\|\nabla\cdot\|_2.$ Hereafter we denote  by $\|\cdot\|_q$ the norm in $L^q(\RT)$, for $q\in [1,+\infty]$.

We are looking for {\em weak solutions} in the following sense.
\begin{definition}\label{def:ws}
A {\em weak solution} of \eqref{eq} is a couple $(u,\phi)\in \H \times \X$
such that for all $(v,\psi) \in C^{\infty}_{c}(\RT)\times C^{\infty}_{c}(\RT)$, we have
\begin{equation*}
\begin{cases}
\dis  \irt \n u \cdot \n v+u v
+\phi u v
=\irt |u|^{p-1}uv
\\[5mm]
\dis \irt \frac{\n \phi\cdot \n \psi}{\sqrt{1-|\nabla \phi|^2}}
=\irt u^2 \psi.
\end{cases}
\end{equation*}
\end{definition}
Observe that the boundary condition at infinity is encoded in the functional space.

Of course, the fact that the setting $\H\times\X$ is not a vector space is a nontrivial obstacle to our variational approach. In particular, to compute variations with respect to $\phi$ along the direction established by a generic smooth and compactly supported function, we need to require in advance that $\|\n \phi\|_\infty< 1$. This fact brings with it a concrete complication, for example in dealing with the reduction method which is a standard tool used in this kind of problems (see, for example, \cite{BF1,BF2,yu}). Indeed, the strongly indefinite nature of the functional can be classically removed showing that, for any $u\in \HT$ fixed, there exists a unique $\phi_u\in \X$ solution of the second equation of system \eqref{eq} and reducing the problem to that of finding critical points of the (no more strongly indefinite) one-variable functional $I(u)=F(u,\phi_u)$, defined on $\HT$ (see Section \ref{se:fs} for more details).
\\
As a consequence, we are led to consider a preliminary minimizing problem on the set $\X$ and then, because of the bad properties of $\X$ itself, we have to study the relation between solutions of this minimizing problem and solutions of the second equation (with respect to $\phi$, being $u$ fixed). This second step is one of the questions left open for \eqref{eqq} in \cite{yu}, which has been recently solved in \cite{BDP} in a radial setting. For this reason, and also in order to overcome difficulties related with compactness, we will restrict our study to radial solutions.
So, let us introduce our functional framework: we set
\[
\HTr=\{u\in \HT \mid u \hbox{ is radially symmetric}\}
\]
and
\[
\X_r=\{\phi\in \X \mid \phi \hbox{ is radially symmetric}\}.
\]

Our main results are the following


\begin{theorem}\label{main}
For any $p\in (5/2,5)$, the problem
\eqref{eq} possesses a radial ground state solution, namely a
solution  $(u,\phi)\in \HTr\times \X_r$ minimizing the functional $F$ among all the nontrivial radial
solutions. Moreover  both $u$ and $\phi$ are of class $C^2(\RT)$.
\end{theorem}

What immediately stands out is the unusual range where $p$ varies. It follows from the fact that, in view of the application of the mountain pass theorem,  we need to find a point with a sufficiently large norm where the functional is negative. In order to do this, usually one computes the reduced one-variable functional $I$ on curves of the type
$$t\in (0,+\infty)\mapsto u_t:=t^\a u(t^\b \cdot)\in\HTr,$$
and look for suitable values of $\a $ and $\b $ for which $I(u_t)<0$ for large values of $t$.
However, in our case,  because of the lack of homogeneity and since a precise expression of $\phi_{u_t}$ is not available, we need to proceed by means of estimates of $\phi_{u }$ which lead, as a consequence, to lose something in terms of powers $p$. 

Summing up, our aim in this paper is to propose 
the new model problem \eqref{eq}
and give a positive answer
concerning the existence of solutions, at least for $p\in (5/2, 5)$.
We leave as an open problem the case of smaller $p$ and the existence of non-radial solutions.

The paper is organized as follows: in Section \ref{se:fs} we introduce the functional setting and present some preliminary results, while in Section \ref{se:pr } we prove Theorem \ref{main}.

We finish this section with some notations. In the following we denote by $\|\cdot\|$ the norm in $\HT$ and by $c,c_i,C,C_i$   arbitrary fixed positive constants which can vary from line to line.

\section{Functional setting and preliminary results}\label{se:fs}

We start recalling some properties of the ambient space $\X$ defined in \eqref{eq:spaceX}.

\begin{lemma}[Lemma 2.1 of \cite{BDP}]
\label{lemma21}
The following assertions hold:
\begin{enumerate}[label=(\roman*),ref=\roman*]
\item \label{it:w1p}$\X$ is continuously embedded in $W^{1,p}(\RT)$, for all $p\in[6,+\infty)$;
\item \label{it:embLinf}  $\X$ is continuously embedded in $L^\infty(\RT)$;
\item \label{it:C0} if $\phi\in \X$, then $\lim_{|x|\to \infty} \phi(x)=0$;
\item \label{it:wc} $\X$ is weakly closed;
\item \label{it:compact}  if $(\phi_n)_n\subset\X$ is bounded, there exists $\bar \phi\in \X$ such that, up to a subsequence, $\phi_{n}\rightharpoonup \bar \phi$ weakly in $\X$ and uniformly on compact sets.
\end{enumerate}
\end{lemma}

As already observed in the Introduction, the functional $F$ is strongly indefinite on $\HT\times \X$ from above and from below, and so we will consider a reduced one-variable functional, solving the second equation of \eqref{eq}, for any fixed $u\in \Hr$. Let us start considering the functional $E:H^{1}(\mathbb R^{3})\times \X \to \R$ defined as
\begin{equation*}
E(u,\phi) = \irt \left ( 1- \sqrt{1-|\nabla \phi|^{2}}\right) -\irt\phi u^{2}.
\end{equation*}

The following lemma holds.

\begin{lemma}\label{le:phiu}
For any $u\in \HT$ fixed, there exists a unique $\phi_u\in \X$ such that the following properties hold:
\begin{enumerate}[label=(\roman*),ref=\roman*]
\item \label{it:min} $\phi_u$ is the unique minimizer of the functional $E(u,\cdot):\X\to \R$ and $E(u,\phi_u)\le 0$, namely
\begin{equation}\label{eq:eneg}
\irt \phi_u u^2
\ge \irt \left(1-\sqrt{1-|\n \phi_u|^2}\right);
\end{equation}
\item \label{it:pos}$\phi_u\ge 0$ and $\phi_u=0$ if and only if $u=0$;
\item \label{it:6} if $\phi$ is a weak solution of the second equation of system \eqref{eq}, then $\phi=\phi_u$ and it satisfies the following equality
\begin{equation}\label{eq:phineh}
\irt \dfrac{|\n \phi_u|^2}{\sqrt{1-|\n \phi_u|^2}}=\irt \phi_u u^2.
\end{equation}
\end{enumerate}
Moreover, if $u\in\HTr$, then $\phi_u\in\X_r$ is the unique weak solution of the second equation of system \eqref{eq}.
\end{lemma}
\begin{proof}
Points \eqref{it:min}, \eqref{it:pos} and \eqref{it:6} are an immediate consequence of Theorems 1.3 and Lemma 2.12 of \cite{BDP}. For the second part of the statement we refer to \cite[Theorem 1.4]{BDP}.
\end{proof}


\begin{remark}\label{rem:continua}
We point out that, as stated in \cite[Remark 5.5]{BDP}, if $w_{n}\to w $ in $L^{p}(\mathbb R^{3})$,
with $p\in[1,+\infty)$ then $\phi_{w_{n}}\to \phi_{w}$ in $L^{\infty}(\mathbb R^{3})$.

\end{remark}

By Lemma \ref{le:phiu}, we can deal with the following one-variable functional defined on $\HT$ as
\begin{align*}
I(u)&=F(u,\phi_u)
\\
&=\frac 12 \irt \left(|\n u|^2+u^2\right)
+\frac 12 \irt \phi_u u^2
-\frac 1{p+1}\irt |u|^{p+1}
-\frac 12 \irt \left(1-\sqrt{1-|\n \phi_u|^2}\right)\\
&= \frac 12 \irt \left(|\n u|^2+u^2\right) -\frac 1{p+1}\irt |u|^{p+1} - \frac{1}{2}E(u,\phi_{u}).
\end{align*}

\begin{proposition}\label{prop:IC1}
The functional $I$ is of class $C^{1}$ and for every $u, v\in H^{1}(\mathbb R^{3})$,
$$I'(u)[v] = \int_{\mathbb R^{3}} \nabla u \cdot \nabla v +\int_{\mathbb R^{3}} uv +\int_{\mathbb R^{3}} \phi_{u} uv
-\int_{\mathbb R^{3}} |u|^{p-1}uv.$$
\end{proposition}
\begin{proof}
Arguing as in \cite{yu}, let us show that
$$I(u+v) - I(u)-DI(u)[v]=o(v),\quad \text{as } v\to 0,$$
where
$$DI(u)[v]:=\int_{\mathbb R^{3}} \nabla u \cdot \nabla v +\int_{\mathbb R^{3}} uv +\int_{\mathbb R^{3}} \phi_{u} uv
-\int_{\mathbb R^{3}} |u|^{p-1}uv
$$
which is trivially linear and continuous in $v$.
\\
We set
$$I(u+v) - I(u)-DI(u)[v] = A_1 + A_2 + A_3$$
where
\begin{align*}
A_1 &:= \frac{1}{2}\int_{\mathbb R^{3}} |\nabla (u+v)|^{2} - \frac{1}{2}\int_{\mathbb R^{3}} |\nabla u|^{2} - \int_{\mathbb R^{3}} \nabla u \cdot \nabla v, \\
A_2&:= -\frac{1}{p+1} \int_{\mathbb R^{3}} |u+v|^{p+1} +\frac{1}{p+1}\int_{\mathbb R^{3}} |u|^{p+1} +\int_{\mathbb R^{3}} |u|^{p-1}uv, \\
A_3 &:= \frac{1}{2}\int_{\mathbb R^{3}} v^{2} -\frac{1}{2}E(u+v,\phi_{u+v})+\frac{1}{2}E(u,\phi_{u}) - \int_{\mathbb R^{3}} \phi_{u}uv.
\end{align*}
Clearly
$$A_1=o(v),\quad A_2=o(v).$$
Now observe that by point \eqref{it:min} of Lemma \ref{le:phiu} we have $E(u+v,\phi_{u})\ge E(u+v,\phi_{u+v})$, so that
an explicit computation gives
\begin{align*}
A_3&\ge\frac{1}{2}\int_{\mathbb R^{3}} v^{2}-\frac{1}{2}E(u+v,\phi_{u})+\frac{1}{2}E(u,\phi_{u})-\int_{\mathbb R^{3}} \phi_{u}uv \nonumber \\
&=\frac{1}{2}\int_{\mathbb R^{3}} v^{2} +\frac{1}{2}\int  \phi_{u} v^{2}\nonumber
\ge \frac{1}{2}\int_{\mathbb R^{3}} \phi_{u}  v^{2} \nonumber
=o(v)
\end{align*}
being $\displaystyle\Big|\int_{\mathbb R^{3}} v^{2} \phi_{u}\Big| \le C \|v\|^{2} \|\phi_{u}\|_{\infty}.$
Analogously, once again by point \eqref{it:min} of Lemma \ref{le:phiu}, being  $E(u,\phi_{u})\le E(u,\phi_{u+v})$, we get
\begin{align*}
A_3&\le \frac{1}{2}\int_{\mathbb R^{3}} v^{2} -\frac{1}{2}E(u+v,\phi_{u+v})+\frac{1}{2}E(u,\phi_{u+v}) - \int_{\mathbb R^{3}} \phi_{u}u v\\
&=\frac{1}{2}\int_{\mathbb R^{3}}v^{2} + \int_{\mathbb R^{3}} \phi_{u+v}uv +\frac{1}{2} \int_{\mathbb R^{3}} \phi_{u+v}v^{2} - \int_{\mathbb R^{3}}\phi_{u} uv\\
&= o(v) +\frac12 \int \phi_{u+v}v^{2} +\int_{\mathbb R^{3}} (\phi_{u+v} -\phi_{u})uv\\
&\le o(v) + C \|v\|^{2}\|\phi_{u+v}\|_{\infty} +C\|u\| \|v\|\|\phi_{u+v} - \phi_{u}\|_{\infty}
= o(v),
\end{align*}
in view of Remark \ref{rem:continua}.
Hence $A_3=o(v)$ and the differentiability of $I$ is proved.
\\
Finally, let us prove the continuity of the map
$$u\in\HT\mapsto \phi_{u } u\in \mathcal L (H^{1}(\mathbb R^{3}); \mathbb R),$$
from which we easily deduce the continuity of $DI:\HT\to  \mathcal L (H^{1}(\mathbb R^{3}); \mathbb R).$
\\
Let $u_{n}\to u$ in $H^{1}(\mathbb R^{3})$.
Observe that uniformly in $v\in H^{1}(\mathbb R^{3})$, with $\|v\|\le 1$,
\begin{equation*}
 \int_{\mathbb R^{3}} \Big|\phi_{u_{n}} u_{n} -\phi_{u}u  \Big| |v| \le
   \int_{\mathbb R^{3}} |\phi_{u_{n}} |   | u_{n}-u | |v|  + \int_{\mathbb R^{3}} |\phi_{u_{n}} - \phi_{u}||u||v|=o_{n}(1),
\end{equation*}
again by Remark \ref{rem:continua}. The conclusion follows.
\end{proof}


%

\begin{proposition}\label{prop:solution}
If $(u,\phi)\in \HT \times \X$ is a weak nontrivial solution of \eqref{eq}, then $\phi=\phi_u$ and $u$ is a critical point of $I$.
On the other hand, if $u\in\HTr\setminus\{0\}$ is a critical point of $I$, then $(u,\phi_u)$ is a weak nontrivial solution of \eqref{eq}.
\end{proposition}

\begin{proof}
The first part of the statement is a consequence of \cite[Proposition 2.6]{BDP} and Proposition \ref{prop:IC1}, while the second part follows by Lemma \ref{le:phiu} and Proposition \ref{prop:IC1}.
\end{proof}

In the next proposition we are going to prove that $\HTr$ is a natural constraint for the functional $I$.

\begin{proposition}\label{pr:criticalita}
If $u\in \HTr$ is a critical point of $I_{|\HTr}$, then $u$ is a critical point of $I$.
\end{proposition}

\begin{proof}
Denote by $O(3)$ the group of rotations in $\RT$ and 
for any $g\in O(3)$ consider the action
induced on $H^{1}(\RT)$, that is
	 $$T_g :u\in H^{1}(\RT)\mapsto  u\circ g\in H^{1}(\RT).$$ 
Clearly $\HTr$ is the set of the fixed points for  the group $T=\{T_{g}\}_{g\in O(3)}$
 namely 
\[
\HTr = \{ u\in \HT \mid T_g u = u \  \hbox{ for all } g \in O ( 3 ) \}. 
\]
Then the conclusion can be achieved by the Palais' Principle of Symmetric Criticality, 
if we show that $I$ is invariant under the action of $T$, that is
\begin{equation*}
I(T_g u)=I(u), \quad \hbox{for all } g \in O ( 3 ), u\in \HT. 
\end{equation*}
Actually it is sufficient to show that
$\phi_{T_gu}=T_g\phi_u$ for any $u\in \HT$ and for all $g \in O ( 3 )$.
To this aim, by Lemma \ref{le:phiu}, we have 
\[
E(u,T_{g^{-1}}\phi_{T_gu})
=E(T_gu,\phi_{T_gu})
\le E(T_g u,T_{g}\phi_{u})
=E(u,\phi_{u})
\]
and so, by the uniqueness of the minimizer of $E(u,\cdot)$, we conclude that $\phi_u=T_{g^{-1}}\phi_{T_gu}$ as desired.
\end{proof}

The following technical lemma will be useful to study the geometry of the functional $I$.
\begin{lemma}\label{le:q}
Let $q$ be in $[2,3).$ Then there exist positive constants $C$ and $C'$ such that, for any $u\in \H$, we have
\[
\|\n \phi_u\|_2^{\frac{q-1}{q}}\le C\|u\|_{2(q^*)'}\le C'\|u\|,
\]
where $q^*$ is the critical Sobolev exponent related to $q$ and $(q^*)'$ is its conjugate exponent,  namely
\[
q^*=\frac{3q}{3-q} \ \hbox{ and }\
 (q^*)'=\frac{3q}{4q-3}.
\]
\end{lemma}

\begin{proof}
Since $\|\n \phi_{u}\|_\infty\le 1$ and $q< 3$,
\begin{equation*}
\|\phi_u\|_{q^*}
\le C \|\n \phi_u\|_{q}
=C\left(\irt |\n \phi_u|^2 |\n \phi_u|^{q-2}\right)^\frac 1q
\le C \|\n \phi_u\|_{2}^\frac 2q,
\end{equation*}
so, by \eqref{eq:eneg} and being $2(q^*)'\in [2,6]$, we have
\begin{align*}
\|\n \phi_u\|_2^2
&\le C\irt\left( 1-\sqrt{1-|\n \phi_u|^2}\right)\le C\irt \phi_u u^2
\\
&\le C\|\phi_u\|_{q^*} \|u\|_{2(q^*)'}^2
\le C \|\n \phi_u\|_{2}^\frac 2q  \|u\|_{2(q^*)'}^2
\end{align*}
and we get the conclusion.
\end{proof}

We conclude this section showing that the radial weak solutions of \eqref{eq} are actually classical and satisfy a Pohozaev type identity.

\begin{proposition}\label{pr:rego}
If $(u,\phi)\in\HTr\times\X_r$ is a weak solution of \eqref{eq}, then both $u$ and $\phi$ are of class $C^2(\RT)$.
\end{proposition}

\begin{proof}
Since $u\in \HTr$, by \cite[Theorem 3.2]{BDP} we deduce that $\phi\in C^1(\RT)$. Looking at the first equation in the system and by using a bootstrap argument,  we conclude that $u\in C^2(\RT)$. We define $\varphi:[0,+\infty[\to\R$ such that for any $r\ge 0:$ $\varphi(r)=\phi(|x|)$ where $x\in\RT$ is arbitrarily chosen in such a way that $|x|=r$.
\\
From now on, we proceed as in \cite[Lemma 1, page 329]{BL}. Since $\phi$ is radial and satisfies the second equation in a weak sense, we deduce that
    \begin{equation*}\label{eq:distribution}
        D\left(\frac{\varphi'r^2}{\sqrt{1-| \varphi'|^2}}\right)=-u^2r^2,\quad \hbox{in }(0,+\infty)
    \end{equation*}
where the symbol $D$ denotes the derivative in the sense of distributions.\\
Since on the right hand side we have a continuous function, the derivative actually has to be meant in the classical sense. So, integrating in $(0,r)$ and since $\varphi'(0)=0$,
    \begin{equation}\label{eq:distribution2}
        \frac{\varphi'(r)}{\sqrt{1-| \varphi'(r)|^2}}=-\frac 1{r^2}\int_0^r u^2(s)s^2 \,ds=:f(r)\in C^1\big((0,+\infty)\big).
    \end{equation}
On the one hand, by \eqref{eq:distribution2}, we deduce that, for $r>0$, we have
    \begin{equation*}
        f'(r)=\frac 2{r^3}\int_0^r u^2(s)s^2 \,ds-u^2(r)
    \end{equation*}
and then $\lim_{r\to 0}f'(r)=-\frac 13 u^2(0)$.\\
On the other hand, again by \eqref{eq:distribution2},
    \begin{equation*}
        \lim_{r\to 0}\frac{f(r)}{r}= \lim_{r\to 0}-\frac 1{r^3}\int_0^r u^2(s)s^2 \,ds=-\frac 13 u^2(0).
    \end{equation*}
We conclude that there exists $f'(0)$ and $\lim_{r\to 0}f'(r)=f'(0)$. Then $f\in C^1\big([0,+\infty)\big)$.
By some computations, by \eqref{eq:distribution2}, we have
    \begin{equation*}
        \varphi'(r)= \frac{f(r)} {\sqrt{1+f^2(r)}}\in C^1\big([0,+\infty)\big)
    \end{equation*}
and we are done.
\end{proof}

\begin{proposition}\label{pr:poho}
If $(u,\phi)$ is a solution of \eqref{eq} of class $C^2(\RT)$, then the following Pohozaev type identity is satisfied:
\begin{multline}\label{eq:poho}
\frac 12 \irt |\n u|^2+\frac 32 \irt u^2
+2 \irt \frac{|\n \phi|^2}{\sqrt{1-|\n \phi|^2}}
\\-\frac 32 \irt \left(1-\sqrt{1-|\n \phi|^2}\right)
=\frac{3}{p+1}\irt |u|^{p+1}.
\end{multline}
\end{proposition}

\begin{proof}
Arguing as in \cite{DM2}, for every $R>0$, we have
\begin{align}
\int_{B_R}\!\! -\Delta u (x\cdot\nabla u)&= - \frac{1}{2}\int_{B_R}\!\!|\nabla u|^2-
\frac{1}{R}\int_{\partial B_R} \!\!|x\cdot\nabla u|^2 + \frac{R}{2}\int_{\partial B_R}\!\! |\nabla u|^2, \label{eq:pohoBR1}
\\
\int_{B_R} u (x\cdot\nabla u)&=-\frac 32\int_{B_R} u^2 + \frac R2\int_{\partial B_R} u^2, \label{eq:pohoBR2}
\\
\int_{B_R} \phi u (x\cdot\nabla u)&=- \frac{1}{2} \int_{B_R} u^2 (x\cdot \nabla\phi)- \frac{3}{2}\int_{B_R}\phi u^2
+\frac{R}{2}\int_{\partial B_R} \phi u^2, \label{eq:pohoBR3}
\\
\int_{B_R} |u|^{p-1}u (x\cdot\nabla u)&=-\frac 3{p+1}\int_{B_R} |u|^{p+1} + \frac R{p+1}\int_{\partial B_R} |u|^{p+1}, \label{eq:pohoBR4}
\end{align}
where $B_R$ is the ball of $\RT$ centered in the origin and with radius $R$.
\\
Moreover, denoting by $\d _{ij}$ the Kronecker symbols, since for any $i,j=1,2,3$,
\begin{align*}
\int_{B_R}\!\!\!\de_i\left(\dfrac{\de_i \phi}{\sqrt{1-|\n \phi|^2}}\right)\!x_j \ \de_j \phi
&=-\int_{B_R}\dfrac{\de_i \phi\ \de_j \phi}{\sqrt{1-|\n \phi|^2}} \d_{ij}
-\int_{B_R}\dfrac{\de_i \phi\ \de^2_{i,j}\phi}{\sqrt{1-|\n \phi|^2}} x_j
\\
&\quad+\int_{\de B_R}\dfrac{\de_i \phi\ \de_j \phi}{\sqrt{1-|\n \phi|^2}} \frac{x_i x_j}{|x|},
\end{align*}
we have
\begin{align}
\int_{B_R}\!\!\!\!
-\dv &\left(\dfrac{\n \phi}{\sqrt{1-|\n \phi|^2}}\right) \!(x\cdot \nabla\phi)
=-\sum_{i,j=1}^3 \int_{B_R}\de_i\left(\dfrac{\de_i \phi}{\sqrt{1-|\n \phi|^2}}\right)x_j \de_j \phi \nonumber
\\
&=\int_{B_R}\dfrac{|\n \phi|^2}{\sqrt{1-|\n \phi|^2}}
+\sum_{j=1}^3\int_{B_R}\de_j \left(1-\sqrt{1-|\n \phi|^2}\right) x_j \nonumber
\\
&\quad-\sum_{i,j=1}^3\int_{\de B_R}\dfrac{\de_i \phi\ \de_j \phi}{\sqrt{1-|\n \phi|^2}} \frac{x_i x_j}{|x|} \nonumber
\\
&=\int_{B_R}\dfrac{|\n \phi|^2}{\sqrt{1-|\n \phi|^2}}
-3\int_{B_R} \left(1-\sqrt{1-|\n \phi|^2}\right) \nonumber
\\
&\quad
+R\int_{\de B_R} \left(1-\sqrt{1-|\n \phi|^2}\right)  -\sum_{i,j=1}^3\int_{\de B_R}\dfrac{\de_i \phi\ \de_j \phi}{\sqrt{1-|\n \phi|^2}} \frac{x_i x_j}{|x|}. \label{eq:pohoBR5}
\end{align}
Multiplying the first equation of \eqref{eq} by $x\cdot\nabla u$ and the second equation
by $x\cdot\nabla\phi$ and integrating on $B_R$, by \eqref{eq:pohoBR1}, \eqref{eq:pohoBR2}, \eqref{eq:pohoBR3}, \eqref{eq:pohoBR4} and \eqref{eq:pohoBR5} we get, respectively,
\begin{align}
- \frac{1}{2}\int_{B_R}&\!\!|\nabla u|^2-
\frac{1}{R}\int_{\partial B_R} \!\!|x\cdot\nabla u|^2 + \frac{R}{2}\int_{\partial B_R}\!\! |\nabla u|^2
-\frac 32\int_{B_R} u^2 + \frac R2\int_{\partial B_R} u^2 \nonumber
\\
&\quad- \frac{1}{2} \int_{B_R} u^2 (x\cdot \nabla\phi)- \frac{3}{2}\int_{B_R}\phi u^2
+\frac{R}{2}\int_{\partial B_R} \phi u^2\nonumber
\\
&=-\frac 3{p+1}\int_{B_R} |u|^{p+1} + \frac R{p+1}\int_{\partial B_R} |u|^{p+1}, \label{eq:po1}
\end{align}
and
\begin{align}
\int_{B_R} u^2 (x\cdot \n \phi)
&=\int_{B_R}\dfrac{|\n \phi|^2}{\sqrt{1-|\n \phi|^2}}
-3\int_{B_R} \left(1-\sqrt{1-|\n \phi|^2}\right) \nonumber
\\
&\quad
+R\int_{\de B_R} \left(1-\sqrt{1-|\n \phi|^2}\right)  -\sum_{i,j=1}^3\int_{\de B_R}\dfrac{\de_i \phi\ \de_j \phi}{\sqrt{1-|\n \phi|^2}} \frac{x_i x_j}{|x|}. \label{eq:po2}
\end{align}
Substituting \eqref{eq:po2} into \eqref{eq:po1}, since all the boundary integrals go to zero as $R\to +\infty$ (we can repeat the arguments of \cite{BL}), by \eqref{eq:phineh} we get the conclusion.
\end{proof}

\section{Proofs of the main results}\label{se:pr }

Using an idea from \cite{J,struwe}, we look for bounded Palais-Smale
sequences of the following perturbed functionals
\begin{equation*}
I_\l(u)=\frac 12 \irt (|\n u|^2+u^2)
+\frac 12 \irt \phi_u u^2
-\frac 12 \irt \left(1-\sqrt{1-|\n \phi_u|^2}\right)
-\frac \l{p+1}\irt |u|^{p+1},
\end{equation*}
for almost all $\l$ near $1$. Then we will deduce the existence of
a non-trivial critical point $v_\l$ of the functional $I_\l$ at
the mountain pass level. Afterward, we study the convergence of
the sequence $(v_\l)_\l$, as $\l$ goes to 1 (observe that
$I_1=I$).
\\
We begin applying a slightly modified version of the monotonicity trick
due to \cite{J,struwe}.

\begin{proposition}\label{prop:mt}
Let $\big(X,\|\cdot\|\big)$ be a Banach space and $J\subset\R^+$ an interval.
Consider a family of $C^1$ functionals $I_{\lambda}$ on $X$ defined by
\begin{equation*}
I_\l(u)=A(u)- \l B(u), \qquad \hbox{for} \ \l\in J,
\end{equation*}
with $B$ non-negative and either $A(u)\to + \infty$ or
$B(u)\to+\infty$ as $\|u\|\to+\infty$ and such that $I_\l(0)=0$.
For any $\l\in J$, we set
\begin{equation*}
\Gamma_\l:=\{\gamma\in C([0,1],X)\mid \gamma(0)=0, \ I_\l(\gamma(1))< 0\}.
\end{equation*}
Assume that for every $\l\in J$, the set $\G_\l$ is non-empty and
\begin{equation*} \label{eq:cl}
c_\l:=\inf_{\gamma\in\Gamma_\l}\max_{t\in[0,1]} I_\l(\gamma(t)) >0.
\end{equation*}
Then for almost every $\l\in J$, there is a sequence $(v_n)_n  \subset X$ such that
\begin{itemize}
\item[\rm(i)] $(v_n)_n $ is bounded in $X$;
\item[\rm(ii)] $I_\l(v_n)\to c_\l$, as $n\to +\infty$;
\item[\rm(iii)] $I_\l'(v_n)\to 0$ in the dual space $X^{-1}$ of $X$, as $n\to +\infty$.
\end{itemize}
\end{proposition}

In our case $X=\HTr$
\begin{align*}
A(u)&=\frac 12 \irt (|\n u|^2+u^2)
+\frac 12 \irt \phi_u u^2
-\frac 12 \irt \left(1-\sqrt{1-|\n \phi_u|^2}\right),
\\
B(u)&=\frac 1{p+1}\irt |u|^{p+1}.
\end{align*}
Observe that, by \eqref{eq:eneg}, $A(u)\to +\infty$ as $\|u\|\to +\infty$.

\begin{proposition}\label{pr:gamma}
For all $\l \in [1/2,1]$, the set $\G_\l $ is not empty.
\end{proposition}

\begin{proof}
Fix $\l \in [1/2,1]$ and $u\in \HTr\setminus\{0\}$, then, by Lemma \ref{le:q} and for $q\in [2,3)$, we have
\begin{align*}
I_\l(u)&\le \frac 12 \| u\|^2
+\frac 12 \irt \phi_u u^2
-\frac \l{p+1}\|u\|_{p+1}^{p+1}
\\
&\le \frac 12 \| u\|^2
+c \| \phi_u\|_{6} \|u\|_{\frac{12}{5}}^2
-\frac \l{p+1}\|u\|_{p+1}^{p+1}
\\
&\le \frac 12 \| u\|^2
+c \|\n \phi_u\|_{2} \|u\|^2
-\frac \l{p+1}\|u\|_{p+1}^{p+1}
\\
&\le \frac 12 \| u\|^2
+c \|u\|^{\frac{3q-2}{q-1}}
-\frac \l{p+1}\|u\|_{p+1}^{p+1}.
\end{align*}
Therefore, if $\l \in [1/2,1]$ and $u\in \HTr\setminus\{0\}$ and $t>0$, we infer that
\[
I_\l(tu)\le c_1 t^2 +c_2 t^{\frac{3q-2}{q-1}}-c_3 \l t^{p+1}.
\]
Since $p\in (5/2,5)$, we can find $q\in[2,3)$ such that $I_\l(tu)<0$, for $t$ sufficiently large.
\end{proof}

\begin{proposition}\label{pr:a}
For any $\l \in [1/2,1]$, there exist $\a>0$ and $\rho>0$, sufficiently small, such that $I_\l(u)\ge \a $, for all $u\in H^{1}(\mathbb R^{3})$, with $\|u\|=\rho$.
As a consequence $c_\l \ge \a $.
\end{proposition}

\begin{proof}
The conclusion follows easily by Lemma \ref{le:phiu}.
\end{proof}

\begin{proposition}\label{pr:ael}
For almost every $\l \in J$, there exists $u_\l \in \HTr$, $u_\l\neq 0$, such that $I'_\l(u_\l) = 0$ and $I_\l(u_\l) = c_\l $.
\end{proposition}

\begin{proof}

By Propositions \ref{pr:gamma} and \ref{pr:a} we can apply the monotonicity trick (Proposition~\ref{prop:mt}) and we argue that, for almost every $\l \in J$ there exists a bounded Palais-Smale sequence $(u_n)_n\subset \HTr$ for the functional $I_\l $ at level $c_\l $, namely as $n \to +\infty$,
\[
I_\l(u_n)\to c_\l, \qquad I_\l'(u_n)\to 0.
\]
Fix such a $\l \in J$. Exploiting compactness results holding for $\HTr,$ we have that there exists $u_\l\in\HTr$
such that, up to subsequences,
\begin{align}
u_n\rightharpoonup u_\l\;&\hbox{ weakly in }\HTr,\label{eq:weak}
\\
u_n\to u_\l\;&\hbox{ in }L^s(\RT),\; 2<s<6, \nonumber 
\\
u_n\to u_\l\;&\hbox{ a.e. in }\RT. \nonumber 
\end{align}
By \cite[Remark 5.5]{BDP}, we infer that $\phi_n:=\phi_{u_n} \to \phi_{u_\l}=:\phi_\l$, weakly in $\D$ (and uniformly in $\RT$) so we conclude that, for every $v\in\HTr,$
	$$\lim_n I'_\l(u_n)[v]=I'_\l(u_\l)[v]=0$$
that is $u_\l$ is a critical point of $I_\l$.\\
Moreover, since the following convergence holds
\begin{equation}\label{eq:convergence}
\left|\irt \phi_n u_n^2 - \irt \phi_\l u_\l ^2\right|
\le \left|\irt \phi_n (u_n^2 -  u_\l ^2)\right|
+\left|\irt (\phi_n  - \phi_\l) u_\l ^2\right| \xrightarrow[n\to +\infty]{} 0,
\end{equation}
taking into account that $I'_\l(u_n)[u_n]=o_n(1)$ and $I'_\l(u_\l)[u_\l]=0$, by Proposition \ref{prop:IC1} it follows
\begin{align*}
\lim_n \|u_n\|^2
=\lim_n \left(\irt |u_n|^{p+1}-\irt \phi_{u_n}u_n^2\right)
=\irt |u_\l|^{p+1}-\irt \phi_{\l}u_\l^2
=\|u_\l\|^2.
\end{align*}
By this and \eqref{eq:weak} we deduce that $u_n\to u_\l$ in $\HTr$ and then, by \eqref{eq:convergence},
$$0<c_\l=\lim_n I_\l(u_n)=I_\l(u_\l)$$
which concludes the proof.
\end{proof}

Now we are ready to prove our main result.

\begin{proof}[Proof of Theorem \ref{main}]
By Proposition \ref{pr:ael}, there exists a sequence $(\l_n)_n\subset J$ such that $\l_n \nearrow 1$ and, for all $n\in \N$, there exists $u_n\in \HTr\setminus \{0\}$ such that
\begin{align}
& I_{\l_n}(u_n)=c_{\l_n}, \label{eq:cln}
\\
& I_{\l_n}'(u_n)=0 \quad \hbox{in }(H^{1}(\mathbb R^{3}))'. \nonumber
\end{align}
For the sake of brevity, we will denote $\phi_n:=\phi_{u_n}$.  By \eqref{eq:phineh}, \eqref{eq:poho} and since $I'_{\l_n}(u_n)[u_n]=0$, we have
\begin{align*}
\frac 12 \irt |\n u_n|^2+\frac 32 \irt u_n^2
+2 \irt \frac{|\n \phi_n|^2}{\sqrt{1-|\n \phi_n|^2}}
-\frac 32 \irt \left(1-\sqrt{1-|\n \phi_n|^2}\right)
&=\frac{3\l_n}{p+1}\irt |u_n|^{p+1}
\\
\irt |\n u_n|^2+\irt u_n^2
+ \irt \frac{|\n \phi_n|^2}{\sqrt{1-|\n \phi_n|^2}}
&=\l_n\irt |u_n|^{p+1}.
\end{align*}
Multiplying the first equation by $\a /3$ and the second one by $\b /(p+1)$ and summing, we have
\begin{multline*}
\frac{(\a +\b)\l_n}{p+1}\irt |u_n|^{p+1}
=\left(\frac \a6 +\frac{\b}{p+1}\right) \irt |\n u_n|^2
+\left(\frac \a2 +\frac{\b}{p+1}\right) \irt u_n^2
\\
\quad+\left(\frac{2\a}3 +\frac{\b}{p+1}\right) \irt \frac{|\n \phi_n|^2}{\sqrt{1-|\n \phi_n|^2}}
-\frac \a2 \irt \left(1-\sqrt{1-|\n \phi_n|^2}\right).
\end{multline*}
Assuming, in particular, $\a =1-\b $, we get
\begin{multline}
\frac{\l_n}{p+1}\irt |u_n|^{p+1}
=\left(\frac 16 +\frac{\b(5-p)}{6(p+1)}\right) \irt |\n u_n|^2
+\left(\frac 12 +\frac{\b(1-p)}{2(p+1)}\right) \irt u_n^2
\\
\quad+\left(\frac{2}3 +\frac{\b(1-2p)}{3(p+1)}\right) \irt \frac{|\n \phi_n|^2}{\sqrt{1-|\n \phi_n|^2}}
-\left(\frac 12-\frac \b2\right) \irt \left(1-\sqrt{1-|\n \phi_n|^2}\right). \label{eq:unp1}
\end{multline}
Therefore, since for all $t\in [0,1[$
\[
1-\sqrt{1-t}\le \frac 12 \frac{t}{\sqrt{1-t}},
\]
substituting \eqref{eq:unp1} into \eqref{eq:cln}, we have
\begin{align*}
c_{\ln}&=I_{\ln}(u_n)
=\left(\frac 13 -\frac{\b(5-p)}{6(p+1)}\right) \irt |\n u_n|^2
+\frac{\b(p-1)}{2(p+1)}\irt u_n^2
\\
&\quad+\left(\frac{\b(2p-1)}{3(p+1)}-\frac{1}6 \right) \irt \frac{|\n \phi_n|^2}{\sqrt{1-|\n \phi_n|^2}}
-\frac \b2 \irt \left(1-\sqrt{1-|\n \phi_n|^2}\right)
\\
&\ge \left(\frac 13 -\frac{\b(5-p)}{6(p+1)}\right) \irt |\n u_n|^2
+\frac{\b(p-1)}{2(p+1)}\irt u_n^2
\\
&\quad+\left(\frac{\b(2p-1)}{3(p+1)}-\frac{1}6 -\frac \b4 \right) \irt \frac{|\n \phi_n|^2}{\sqrt{1-|\n \phi_n|^2}}.
\end{align*}
Since $p>2$, there exists a constant $\b $ such that all the coefficients in the previous inequality are positive and so, by the boundedness of $(c_{\ln})_n $ (indeed the map $\l \mapsto c_\l$ is non-increasing), we infer the boundedness of the sequence $(u_n)_n$ in $\HTr$, too.
\\
Now, arguing similarly as in the proof of Proposition \ref{pr:ael}, we can easily prove the existence of a nontrivial critical point $u$ of $I$. Hence we have
\[
\mathcal{S}_r:=\left\{u\in\HTr\setminus\{0\}\mid
        I'(u)=0\right\}\neq \emptyset.
\]
Moreover, any $u\in \mathcal{S}_r$ satisfies
\[
\|u\|^2\le \irt |\n u|^2+\irt u^2
+ \irt \frac{|\n \phi_u|^2}{\sqrt{1-|\n \phi_u|^2}}
=\irt |u|^{p+1}\le C\|u\|^{p+1},
\]
and therefore
\[
\inf_{u\in \mathcal{S}_r}\|u\|>0.
\]
Since we have that $I(u)\ge c \|u\|^2$ for all $u \in \mathcal{S}_r,$ we conclude that
\[
\sigma_r:=\inf_{u\in \mathcal{S}_r} I(u)>0.
\]
Let $(u_n)_n\subset\mathcal{S}_r$ such that $I(u_n)\to \sigma_r.$
Arguing as before we have that the
sequence is bounded. Finally, as in the proof of Proposition \ref{pr:ael}, there exists $u\in\HTr$ critical point of $I$
such that, up to subsequences, $u_n\to u$ in $\HTr$. Then $(u,\phi_u)$ is a radial ground state solution by Proposition \ref{prop:solution}.
\\
Finally, by Proposition~\ref{pr:rego} we conclude that $u$ and $\phi_u$ are of class $C^2(\RT)$.
\end{proof}

\end{document}